\documentclass[letterpaper,11pt,]{article}

\usepackage{amsthm,amsmath,amssymb,fullpage}
\usepackage{mathrsfs}
\usepackage[
letterpaper,
top=1in,
bottom=1in,
left=1in,
right=1in]{geometry}
\allowdisplaybreaks
\usepackage[singlespacing]{setspace}
\usepackage[ruled,vlined]{algorithm2e}

\usepackage[colorlinks,citecolor=blue,linkcolor=magenta,bookmarks=true]{hyperref}
\usepackage[nameinlink,capitalise]{cleveref}

\usepackage{thm-restate}

\usepackage[dvipsnames]{xcolor}

\usepackage{amsfonts}
\usepackage{dsfont}

\theoremstyle{definition}

\crefname{equation}{equation}{equations}
\crefname{lem}{lemma}{lemmata}
\crefname{thm}{theorem}{theorems}
\crefname{prop}{proposition}{propositions}
\crefname{cor}{corollary}{corollaries}
\crefname{rem}{remark}{remarks}
\crefname{defn}{definition}{definitions}
\crefname{quest}{question}{questions}

\usepackage[utf8]{inputenc} 
\usepackage[T1]{fontenc}
\usepackage{graphicx,colordvi,psfrag}
\usepackage{amsmath,amssymb}
\usepackage{epstopdf}
\usepackage{calc,pstricks, pgf, xcolor}
\usepackage{nicefrac}
\usepackage{enumerate}
\usepackage{dsfont}
\usepackage{bm}
\usepackage{color}
\usepackage{bbold}
\usepackage{setspace}
\usepackage{amsmath}%

\newtheorem{corollary}{Corollary}
\newtheorem{remark}{Remark}

\newtheorem{definition}{Definition}
\newtheorem{proposition}{Proposition}

\newcommand{\Indep}{\perp \!\!\! \perp}

\newcommand{\Bin}{\mathrm{Binomial}}

\newcommand{\Unif}{\mathrm{Uniform}}

\newcommand{\indic}{\mathds{1}_}

\def\Expt{\mathbb{E}}
\def\Prob{\mathbb{P}}

\def\a{\alpha}

\def\d{\delta}
\def\e{\epsilon}

\def\p{\psi}

\def\v{\nu}

\def\P{{\mathsf P}}
\def\E{\mathbb{E}}
\def\N{{\mathrm N}}

\def\Real{{\mathbb R}}
\def\Nat{{\mathbb N}}

\def\Real{\mathbb{R}}

\def\ph{\widehat{p}}

\def\Multinomial{\mathrm{Multinomial}}
\def\lf{\lfloor}
\def\rf{\rfloor}

\title{Sharp Concentration Inequalities for the Centered Relative Entropy}

\author{
Alankrita Bhatt\\
University of California, San Diego \\
{\tt a2bhatt@eng.ucsd.edu}\\
\and
Ankit Pensia\thanks{Supported by NSF grants DMS-1749857 and CCF-1841190.}\\
University of Wisconsin-Madison\\
{\tt ankitp@cs.wisc.edu}\\
}
\date{\vspace{-7ex}}

\begin{document}
\maketitle

\begin{abstract}
  We study the relative entropy between the empirical estimate of a discrete distribution and the true underlying distribution.
  If the minimum value of the probability mass function exceeds an $\a > 0$ (i.e. when the true underlying distribution is bounded sufficiently away from the boundary of the simplex), we prove an upper bound on the moment generating function of the \emph{centered} relative entropy that matches (up to logarithmic factors in the alphabet size and $\a$) the optimal asymptotic rates, subsequently leading to a sharp concentration inequality for the centered relative entropy. As a corollary of this result we also obtain confidence intervals and moment bounds for the centered relative entropy that are sharp up to logarithmic factors in the alphabet size and $\a$.
\end{abstract}

\section{Introduction}\label{sec:intro}
Let $p:= (p_1,\dots,p_k) \in \Delta^{k-1}$ be a distribution on $k$ alphabets, where $\Delta^{k-1}$ denotes the $k-1$ dimensional simplex for some $k \in \Nat\setminus \{1\}$. Consider $X_1,\dotsc,X_k \sim \Multinomial(n,p)$ for $n \in \Nat$. Define 
\begin{align}\label{eq:PhatDefns}
\ph_{n,k} = (\ph_1,\dotsc,\ph_k) := \left(\frac{X_1}{n},\dotsc,\frac{X_k}{n}\right)
\end{align}
as the empirical estimate of the distribution $p$. Observe that $\ph_{n,k}$ corresponds to the empirical estimate of $p$ constructed using $n$ i.i.d. samples from $p$.  In this paper, we study the concentration of the Kullback--Leibler (KL) divergence (also known as the relative entropy) between $\ph_{n,k}$ and $p$,
\[
D(\ph_{n,k}\|p) := \sum_{i=1}^k \ph_i \log \frac{\ph_i}{p_i},
\]
around its mean. More concretely, for any $\e > 0$ we aim to obtain concentration inequalities of the form 
\begin{align}\label{eq:ConcIneqConventional}
\Prob\left(\left|D(\ph_{n,k}\|p) - \Expt[D(\ph_{n,k}\|p)]\right|\ge \e \right) \le \rho_{n,k}(\e)
\end{align}
for a function $\rho_{n,k}(\cdot)$, for any distribution $p$.
For ease of exposition, we will state the results after re-scaling: let $Z_{n,k,p} := 2nD(\ph_{n,k}\|p)$. $Z_{n,k,p}$ has much operational significance since it is the log-likelihood ratio statistic for hypothesis testing when the underlying distribution is a $k$-alphabet multinomial. The Neyman--Pearson lemma~\cite{Lehmann--Romano2006} states that for a fixed significance level, the likelihood ratio test (which returns a hypothesis based on whether or not $Z_{n,k,p}$ exceeds a given threshold) is the most powerful thereby justifying its use. A sharp confidence interval for $Z_{n,k,p}$ therefore has several applications in statistical problems, for example~\cite{Nowak--Tanczos2019, Malloy--Tripathy--Nowak2020}. 

It is a well known result~\cite[Chapter 3]{VanDerVaart2000},~\cite[Lemma 11.1]{Polyanskiy--Wu14} that for any fixed $k\geq 2$ and $p$ with $ \min_{i \in \{1,\dots,k\}}p_i >0$  
,
\begin{align}
\label{eq:ConvInDis}
    Z_{n,k,p} \stackrel{(d)}{\to} \chi^2_{k-1} \,\, \text{as $n \to \infty$},
\end{align} where $\stackrel{(d)}{\to}$ denotes convergence in distribution. This squares well with the intution that any $f-$divergence is ``locally'' $\chi^2$ like---see for example~\cite[Theorem 4.1]{Csiszar--Shields2004} and~\cite[Section 4.2]{Polyanskiy--Wu14}.
Given the importance of KL divergence in practice, it is a natural question to ask how quickly $Z_{n,k,p}$ inherits the structure of $\chi^2_{k-1}$.
An exciting line of recent work initiated in~\cite{Mardia--Jiao--Tanczos--Nowak--Weissman2020} has shown that the right tails of $Z_{n,k,p}$ satisfies strong concentration properties similar to $\chi^2_{k-1}$ as soon as $n \gtrsim k$ for \emph{all} distributions on $k$ alphabets\footnote{We use $a \gtrsim b$ to mean that $a \geq C b$ for some fixed absolute constant $C>0$. Similarly, we use $a \lesssim b$ when $b \gtrsim a$.}.
In particular, Agrawal~\cite{Agrawal2020} showed that for all $n \gtrsim k$ and $\delta \in (0,1)$: with probability $1 - \delta$,
\begin{align}
    Z_{n,k,p} \lesssim k + \log(1/\delta).
\label{EqAgrawalRaw}
\end{align}
This closely resembles the confidence intervals of the raw (i.e. non-centered) $\chi^2_{k-1}$ distribution and is tight up to constants.
Still, this result does not convey the full picture. Indeed, $\chi^2_{k-1}$ satisfies a stronger concentration inequality around its mean (see, for example, $\cite{Boucheron--Lugosi--Massart13}$): with probability $1 - \delta$,
\begin{align}
  |\chi^2_{k-1} - \Expt[\chi^2_{k-1}]| \lesssim  \sqrt{k \log(1/\delta)} + \log(1/\delta).   
  \label{EqChiSq}
\end{align}
We call $k$ the alphabet size. In the large alphabet regime, the difference between \eqref{EqAgrawalRaw} and \eqref{EqChiSq} is significant whenever $\delta$ is not too small, i.e., $\log(1/\delta) \lesssim k$.
Moreover, an inequality of the form~\eqref{EqAgrawalRaw}  does not provide any non-trivial lower bound estimate for $Z_{n,k,p}$.
To the best of our knowledge, the tightest non-trivial bound available for $Z_{n,k,p} - \Expt[Z_{n,k,p}]$ that scales with $\sqrt{k}$ is the following bound due to \cite{Mardia--Jiao--Tanczos--Nowak--Weissman2020}: with probability $1- \delta$, $Z_{n,k,p} - \Expt[Z_{n,k,p}] \lesssim \sqrt{k/\delta}$.   Mardia et al.~\cite{Mardia--Jiao--Tanczos--Nowak--Weissman2020} achieve this result by showing that the standard deviation of $Z_{n,k,p}$ is at most $O(\sqrt{k})$ and invoking the Chebyschev inequality. This naturally leads to the question of characterizing the dependence of higher-order moments of $Z_{n,k,p} - \E[Z_{n,k,p}]$ on $k$ and possibly leveraging those to achieve tighter concentration inequalities than the one obtained by the Chebyschev inequality via the standard deviation bound.

Multiple works have posed the concentration of $Z_{n,k,p}-\E[Z_{n,k,p}]$ as an open problem~\cite{Mardia--Jiao--Tanczos--Nowak--Weissman2020,Agrawal2020, Agrawal2020Thesis}, since the techniques used to prove concentration of the uncentered $Z_{n,k,p}$ do not apply directly to the centered version.
In this work, we take the first steps in deriving exponential concentration inequalities for $Z_{n,k,p} - \Expt[Z_{n,k,p}]$. 
Our main result is a near-optimal (up to logarithmic factors) concentration inequality for the $Z_{n,k,p}$ around its mean for distributions that are bounded away from the edges of the simplex (\cref{thm:MainThm}).

\subsection{Prior Work}\label{subsec:PriorWork} 
Most of the prior related work has focused on studying concentration inequalities for the \emph{non-centered} KL divergence. The classic Sanov theorem~\cite[Theorem 11.4.1]{10.5555/1146355} establishes 
\begin{align}\label{eq:Sanov}
    \Prob\left(Z_{n,k,p} \ge t \right) \le \binom{n+k-1}{k-1}e^{-t/2}
\end{align}
using the method of types~\cite{Csiszar1998} and shows that this is optimal in the sense that for a fixed $k, \e > 0$
\begin{align}\label{eq:SanovOptimal}
    \lim_{n \to \infty} \frac{1}{n}\log \Prob\left(Z_{n,k,p} \ge 2n \e \right) = \e. 
\end{align}
The main result in Mardia et al.~\cite[Theorem 3]{Mardia--Jiao--Tanczos--Nowak--Weissman2020} greatly improved the dependence of the right hand side of~\eqref{eq:Sanov} on the alphabet size $k$. They also established an $O(k)$ bound on the variance of $Z_{n,k,p}$.
Agrawal~\cite{Agrawal2020} established that 
\begin{align}\label{eq:AgrawalMGF}
\Prob\left(Z_{n,k,p} \ge t \right) \le \left(\frac{et}{2(k-1)}\right)^{k-1} e^{-t/2},
\end{align}
which is sharper than the bound of~\cite{Mardia--Jiao--Tanczos--Nowak--Weissman2020} in certain regimes of $k$ and $n$. The bound~\eqref{eq:AgrawalMGF} was further refined for $k > 2$ by a careful analysis of the moment generating function (MGF) of $Z_{n,k,p}$ in~\cite{Guo--Richardson2020}.
Antos and Kontoyiannis~\cite{AntosKontoyiannis2001} used the bounded differences inequality to prove that for $p = \Unif[k]$ (the uniform distribution over $k$ alphabets), $\|Z_{n,k,p} - \Expt{Z_{n,k,p}} \|_{\psi_2} \lesssim \sqrt{n}\log n$, where $\|\cdot\|_{\psi_2}$ denotes the sub-Gaussian norm~\cite{vershynin2018high}. This gives the right rates (up to log factors) when $n \lesssim k$ but gets worse as $n$ increases, and hence does not explain the asymptotic convergence of $Z_{n,k,p}$ in \eqref{eq:ConvInDis}. Our approach is built on the approach of~\cite{Agrawal2020}, as we also use the chain rule of KL divergence (on the MGF) to reduce the alphabet size of the problem. However, employing this method for the centered $Z_{n,k,p}$ requires non-trivial extensions. More details on our approach and the differences from~\cite{Agrawal2020} are provided in \cref{subsec:our_tech}.

\subsection{Our Results}\label{subsec:Results}

Our main result is to show that $Z_{n,k,p}$ satisfies a strong concentration property around its mean. In particular, we show that $Z_{n,k,p} - \E[Z_{n,k,p}]$ belongs to the family of sub-Gamma distributions~\cite{Boucheron--Lugosi--Massart13} (see \cref{sec:prelim} for more background) with certain parameters.
For a random variable $X$, we use $\psi_X(\cdot)$ to denote the logarithm of its MGF, i.e., $\psi_X(t) := \log(\E[e^{tX}])$. We also define for any $\a > 0$, $\Delta^{k-1}_{\a} := \{p \in \Delta^{k-1}:  \min_{i \in \{1,\dots,k\}}p_i \ge \a\}$. We can then state our main result.
\begin{restatable}{theorem}{thmMainThm}
\label{thm:MainThm}
There exists positive constants $C,c$ such that for any $n \in \Nat$, $k \geq 2$, $p \in \Delta^{k-1}_{\a}$ for an arbitrary $\alpha >0$, the following holds\footnote{Note that $\alpha$ may depend on $n$ and $k$.}
\begin{align}\label{eq:SubGammaMGFMain}
     \psi_{Z_{n,k,p}-\E[Z_{n,k,p}]}(t) \le Ck\log^4\left(\frac{k}{\alpha}\right)t^2, \text{ for all } |t| \le \frac{1}{c}.
\end{align}
\end{restatable}
\begin{remark}
Unlike the results of~\cite{Mardia--Jiao--Tanczos--Nowak--Weissman2020} or~\cite{Agrawal2020}, our main theorem involves the term $\a = \min\{p_1,\dotsc,p_k\}$. In particular, we prove a bound on the MGF only for $p \in \Delta^{k-1}_{\a}$---i.e. for all distributions bounded away from the edges of the simplex. Such a term has appeared before in the literature. For example,~\cite[Theorem 12.13]{Boucheron--Lugosi--Massart13} (see also~\cite{Castellan2003,Massart2006}) provides a concentration inequality for Pearson's $\chi^2$ statistic that involves $\min_i p_{i} \ge \a$. Another example is~\cite[Theorem 10]{Kamath--Orlitsky--Pichapati--Suresh2015}, which characterizes the asymptotics of the minmax error in estimating a multinomial distribution from its samples under any $f-$divergence loss. Nevertheless, a concentration inequality such as~\eqref{eq:SubGammaMGFMain} (without involving any $\a$ term) should
hold for distributions with $\min_i p_i$ arbitrarily close to 0 for $n \geq n_0$, where  \textit{ $n_0$ is large enough} (note that $n_0$ may depend on $k$, $p$, and  $\alpha$). This is because we have $Z_{n,k,p} \stackrel{(d)}{\to} \chi^2_{k-1}$ for any such $p$ (i.e. this is a \emph{distribution-free} result); see \cite[Theorem III.3]{Agrawal2020} for more details. We leave the task of establishing such a result with an explicit dependence on $n_0$ for future work. 
\end{remark}
The proof of \cref{thm:MainThm} is given in \cref{sec:MainCentered} and a brief overview of the proof is given in \cref{subsec:our_tech}.
Using the standard arguments for the concentration of sub-Gamma distributions (see \cref{sec:prelim} for more details), we directly obtain the following corollaries:
\begin{corollary}[Confidence Interval] \label{cor:ConfIntCentered}
Under the condition of \cref{thm:MainThm}, we have the following for all $\delta \geq 0$: with probability $1- \delta$,
\begin{align}\label{eq:confInt}
     |Z_{n,k,p}-\E[Z_{n,k,p}]| \lesssim \sqrt{k \log^4\left(\frac{k}{\alpha} \right) \log\left(\frac{1}{\delta}\right)} + \log\left(\frac{1}{\delta}\right).
\end{align}
In particular, we can take $\rho_{n,\e  }$ in \eqref{eq:ConcIneqConventional} to be $2\exp\left(-\min\left(\frac{\e^2}{v}, \frac{\e}{c}\right)\right)$, where $v \lesssim k \log^4(k/\alpha)$ and $c \lesssim 1$. 
\end{corollary}
For a random variable $X$ and positive integer $m$, we use $\|X\|_m$ to denote $(\E|X|^m)^{1/m}$.
\begin{corollary}[Moment Bounds] \label{cor:MomentBdsCentered}
Under the condition of \cref{thm:MainThm}, we have the following for all $m \in \N$:
\begin{align}\label{eq:CenteredMoments}
     \|Z_{n,k,p}-\E[Z_{n,k,p}]\|_m \lesssim \sqrt{k m \log^4\left(\frac{k}{\alpha} \right)} +  m.
\end{align}\end{corollary}

Finally, we show the tight rates (up to constants) for the raw moments of the random variable $Z_{n,k,p}$. The rates in \cite[Theorem III.2]{Agrawal2020} are incorrectly claimed to be optimal; it is claimed that $\|Z_{n,k,p}\|_m = \Theta(km)$, whereas the right rate of $\|\chi^2_{k-1}\|_m$ is $\Theta(k+m)$. To see this, note that $\E[(\chi^2_{k})^{m}] = 2^{m}\Gamma(m + k/2)/\Gamma(k/2)$, where $\Gamma(\cdot)$ is the gamma function (see, for example, \cite{Simon2006}). The upper bound follows by noting $\Gamma(m + k/2)/\Gamma(k/2) \leq (m + k)^m$, and the lower bound follows by noting that $(\Gamma(m + k/2)/\Gamma(k/2))^{1/m} \geq (\max(k^m,\Gamma(m)))^{1/m} \gtrsim \max(k,m) \gtrsim m + k$. 
However, \cite[Theorem III.2]{Agrawal2020} proves a much weaker bound: $\|Z_{n,k,p}\|_m \lesssim mk$. Nonetheless, the main result of Agrawal can establish the optimal bound on the raw moments and we do so in the following Proposition.

\begin{restatable}
[Raw moment bounds]{proposition}{propRawMomentBds} \label{prop:rawMomentBds}
For any $k \geq 2$, $n\geq 1$, $m \geq 1$ and $p \in \Delta^{k-1}$, we have the following:
$\|Z_{n,k,p}\|_m \lesssim k + m$.
\end{restatable}
The proof of \cref{prop:rawMomentBds} is given in \cref{subsec:MainRaw}. 
 At a high level,  \cite[Theorem III.2]{Agrawal2020} achieves the rate of $O(km)$ for the $m$-th moment by showing that the sub-exponential norm of $Z_{n,k,p}$ is $O(k)$. On the other hand, we control the sub-exponential norm of a truncated and translated version of $Z_{n,k,p}$ to get the tight rates in \cref{prop:rawMomentBds}.
 
\subsection{Our Techniques}\label{subsec:our_tech}
We give an overview of our techniques here. 
We crucially use the chain rule of relative entropy to reduce the $k-$alphabet size problem to a problem with smaller alphabet size. This strategy has been successfully used by~\cite{Mardia--Jiao--Tanczos--Nowak--Weissman2020} (for the moments) and~\cite{Agrawal2020} (for the MGF) to control the right tails of $Z_{n,k,p}$.
Using the chain rule and the fact that $(X_1,\dotsc,X_{k-1})\big| X_k \sim \Multinomial(n-X_k,p')$ where $p' = \left(\frac{p_1}{1-p_k},\dotsc,\frac{p_{k-1}}{1-p_k}\right)$, we can  write the following: 
\begin{align}\label{eq:ChainRuleZ1}
    Z_{n,k,p} = Z_{n,2,p_k} + Z_{n-X_k,k-1,p'},
\end{align}
where $Z_{n,2,p_k}$ is the (scaled) binary relative entropy between $\hat{p}_k$ and $p_k$, and $Z_{n-X_k,k-1,p'}$ is the (scaled) relative entropy between $p'$ and $(\frac{X_1}{n-X_1},\dots,\frac{X_{k-1}}{n-X_1})$; see \cref{subsec:ChainRuleBernPoly} for more details.
 We begin by briefly discussing the technique of \cite{Agrawal2020} that we build on.
 At a high level, \cite{Agrawal2020} proceeds by noting the following decomposition: 
\begin{align}\label{eq:AgrawalChainRuleMGF1}
\E[e^{tZ_{n,k,p}}] = \E[\E[e^{tZ_{n,k,p}}|X_k]] = \E[e^{tZ_{n,2,p_k}}\E[e^{tZ_{n-X_k,k-1,p'}}|X_k]],
\end{align}
and subsequently notes that the problem of bounding the MGF of $Z_{n,k,p}$ can be solved by bounding the MGFs of $Z_{n,2,p_k}$ and $Z_{n-X_k,k-1,p'}$. Thus, they reduce the alphabet size of the problem from $k$ to $k-1$. Applying this reasoning inductively, they finally transform the problem to bounding the MGF of $Z_{n,2,p}$ for all $p \in [0,1]$.

However, this approach hits a natural roadblock for the centered MGF. 
Using~\eqref{eq:ChainRuleZ1}, we see that $Z_{n,k,p} - \E[Z_{n,k,p}] = Z_{n,2,p} - \E[Z_{n,2,p}] + Z_{n-X_k,k-1,p'} -\E[Z_{n-X_k,k-1,p'}]$. Now conditioning on $X_k$ and using an argument similar to~\eqref{eq:AgrawalChainRuleMGF1}, we have the following
\begin{align}
\label{eq:CentredMgfError}
    \E[e^{t(Z_{n,k,p} - \E[Z_{n,k,p}])}] = \E[e^{t(Z_{n,2,p_k} - \E[Z_{n,2,p_k}])}\E[e^{t(Z_{n-X_k,k-1,p'} -\E[Z_{n-X_k,k-1,p'}])}|X_k]].
\end{align}
The induction hypothesis will bound the MGF (conditioned on $X_k$) of $Z_{n-X_k,k-1,p'} - \E[Z_{n-X_k,k-1,p'}|X_k]$, which is different from the expression in \eqref{eq:CentredMgfError} because $\E[Z_{n-X_k,k-1,p'}] \neq \E[Z_{n-X_k,k-1,p'}|X_k]$.
Thus we cannot naively use the same next steps as Agrawal---we would need to introduce the $\E[Z_{n-X_k,k-1,p'}|X_k]$ term somewhere to effectively use induction. 

Defining $g_{k,p}(n) := n\E[D(\ph_{n,k}\|p)]$  for $n \in \Nat$, we note that $g_{k-1,p'}(n-X_k) = \E[Z_{n-X_k,k-1,p'}|X_k]$.
In order to use the induction hypothesis, we begin by the following decomposition:
\begin{align}
    \E\left[e^{t(Z_{n,k,p}-\E[Z_{n,k,p}])}\right] \nonumber
    &= \E\Bigg[e^{t(Z_{n,2,p_k}-\E[Z_{n,2,p_k}])} \cdot e^{2t\left(g_{k-1,p'}(n-X_k) - \E[g_{k-1,p'}(n-X_k)]\right)} \nonumber\\
    &\qquad\qquad \cdot  \E\left[ e^{t(Z_{n-X_k,k-1,p'}-\E[Z_{n-X_k,k-1,p'}|X_k])}\Big|X_k\right]\Bigg], 
\end{align}
where we use that $\E[Z_{n-X_k,k-1,p'}] = \E[g_{k-1,p'}(n-X_k)]$ by the tower property of conditional expectation.
Although this allows us to use the induction hypothesis (conditioned on $X_k)$ on $Z_{n-X_k,k-1,p'}-\E[Z_{n-X_k,k-1,p'}|X_k]$, this leads to a new challenge: we need to control the deviation of $g_{k-1,p'}(n-X_k)$ around its mean.
In order to get the right dependence on $k$ in \cref{thm:MainThm}, we need to show that the MGF of centered $g_{k-1,p'}(n-X_k)$ does not scale polynomially with $k$. 
Suppressing the subscripts in notation of $g$ for brevity, we  need to show that $\psi_{g(n-X_k) - \E[g(n-X_k)]}(t) \leq v t^2$, where $v \lesssim \text{polylog}(k/\alpha)$.
As a sanity check, we expect this to be true after large enough $n$ (which may depend on $k$ and $p$) because we know that for a fixed $k$ and $p$, $g_{k,p}(n) = \E[Z_{n,k,p}]/2 \to (k-1)/2$ as $n \to \infty$ (see \eqref{eq:ConvInDis}). Thus after a large enough $n$, we expect $g_{k,p}(n-X_k)$ to be bounded on an interval of constant length, and hence satisfy a subgaussian-style bound.

The technical bulk of our paper is dedicated to showing that $g_{k,p}(\cdot)$ satisfies the required bounds whenever $n \gtrsim k$  (see  \cref{lem:GsubGammaRangeOne,lem:GConcLargeN}). 
In order to prove these results, we crucially use the connection between $g(\cdot)$ and Bernstein polynomials; see \cref{subsec:ChainRuleBernPoly} for more details.

\subsection{Notation and Organization}
In \cref{sec:prelim}, we introduce important concepts and recall some new and old preliminary results. \cref{sec:MainThmPf} is devoted to proving the main results. \cref{sec:LemmasForProof} proves a few concentration inequalities for random variables, that have been used in the main proof. Finally we provide a few concluding remarks in \cref{sec:discussion}.

\underline{Notation:} All logarithms are to the base $e$. We use $d(p\|q)$ for $p,q \in [0,1]$ to denote the binary KL-divergence, so $d(p\|q) := p \log \frac{p}{q} + (1-p) \log \frac{(1-p)}{(1-q)}$. For two random variables $X$ and $Y$, $X \Indep Y$ denotes the independence of $X$ and $Y$.

\section{Mathematical Preliminaries}
\label{sec:prelim}
In this section, we introduce concepts and preliminary results that will be required further ahead. 
\subsection{Sub-Gamma random variables}\label{subsec:SubGamma}
We review the basic theory of sub-Gamma random variables, following~\cite{Boucheron--Lugosi--Massart13}.

\begin{definition}\label{def:SubGamma}
A random variable $X$ is said to be sub-Gamma with variance factor $\v$ and scale parameter $c$, denoted by $X \in \Gamma(\v,c)$, if 
\begin{align}\label{eq:SubGammaDefn}
     \p_{X}(t) &\le \frac{\nu t^2}{2(1-ct)} \text{   for } |t| \le \frac{1}{c}.
\end{align}
\end{definition}
\begin{remark}
Note that \cref{def:SubGamma} implicitly imposes that if $X \in \Gamma(\nu,c)$ with finite $\nu$ and $c>0$, then $X$ has mean zero. This follows by noting that the convexity of $\exp(t x)$ and Jensen inequality imply that $\p_{X}(t) \geq t \E[X]$, violating the condition in \eqref{eq:SubGammaDefn} for some small (in magnitude) $t$ if $\E[X] \neq 0$.
\end{remark}
Note that since $\chi^2_{k-1}$ is a Gamma random variable with shape parameter $k-1$ and scale parameter $2$, we have that $\chi^2_{k-1} - \E[\chi^2_{k-1}] \in \Gamma(2(k-1),2)$. Therefore, in order to match the asymptotics one is required to prove that $Z_{n,k,p} - \E[Z_{n,k,p}] \in \Gamma(\nu, c)$ with $\nu \lesssim k$ and $c \lesssim 1$. 
The definition of sub-Gamma random variables also readily leads to properties regarding the moments and the concentration which we mention in the following proposition.
\begin{proposition}[Theorem 2.3 of ~\cite{Boucheron--Lugosi--Massart13}]\label{prop:SubGTailsMoments}
If $X \in \Gamma(\v,c)$, then
\begin{align}\label{eq:SubGammaTail}
    \Prob(|X| > \sqrt{2 \v t} + ct) \le 2e^{-t},
\end{align}
and for all integers $q \ge 1$ %
\begin{align}\label{eq:SubGammaMoments}
    \E[X^{2q}] \le q!(8 \v)^{q} + (2q)!(4c)^{2q}, \,\,\,\, \text{and}  \,\,\,\,    \|X\|_q \lesssim \sqrt{q\nu} + qc.
\end{align}
\end{proposition}
The usual characterization of sub-Gamma random variables in terms of its tails and moments (i.e., the converse of \cref{prop:SubGTailsMoments}) in the literature requires the random variable to be centered (see, for example, \cite[Theorem 2.3]{Boucheron--Lugosi--Massart13}). We next provide a characterization where $X$ need not be centered\footnote{We note that a similar result appears in \cite[Lemma 1]{BousquetEtAl2019}, but there seems to be a gap in the provided proof; the proof applies \cite[Theorem 2.3]{Boucheron--Lugosi--Massart13}, a result for centered random variable, on a random variable that is not centered. }. Thus, we extend the centering lemma \cite[Lemma 2.6.8]{vershynin2018high} to sub-Gamma random variables.

\begin{restatable}
[Centering and characterization of sub-Gamma]{proposition}{propSubGTailsMomentsII}
\label{prop:SubGTailsMoments2}
Let $X$ be a random variable and let $A,B \geq 0$.
\begin{enumerate}
    \item  Suppose  for every integer $q \geq 1, X$ satisfies
\begin{align}\label{eq:MomentToMGF}
    \E[X^{2q}] \leq q!A^q + (2q)!B^{2q}.
\end{align}
Then, $X - \E[X] \in \Gamma(24A + 36B^2,6B)$. 

\item Suppose for every $\d > 0$ 
\begin{align} \label{eq:UpperBoundToMGF}
    |X| \leq  \sqrt{2A} + B+  \sqrt{2A\log(1/\delta)} + B \log(1/\delta)
\end{align}
holds with probability at least $1-\d$. Then $X - \E[X] \in \Gamma(1536 A + 864B^2, 144B)$.
\end{enumerate}
\end{restatable}
The proof of \cref{prop:SubGTailsMoments2} is deferred to \cref{app:proofprelim}.
\begin{remark}[Centering $Z_{n,k,p}$] \label{rem:CenterAgrawal}
A natural idea in light of \cref{prop:SubGTailsMoments2} is to obtain concentration results for $Z_{n,k,p}-\E[Z_{n,k,p}]$ using the moment bounds on $Z_{n,k,p}$ in \cref{prop:rawMomentBds} and the centering lemma of \cref{prop:SubGTailsMoments2}. Since $\E[Z_{n,k,p}]^{2q} \le (Ck+C(2q))^{2q} $ for an absolute constant $C$, %
we can see that this allows us to establish that $Z_{n,k,p} - \E[Z_{n,k,p}] \in \Gamma(\nu, c)$ where $\nu \lesssim k^2$ and $c \lesssim k$ which does not give us the right rates.
\end{remark}

\subsection{Relative Entropy and Bernstein Polynomials}\label{subsec:ChainRuleBernPoly}
We recall the chain rule of entropy below:
\begin{proposition}[Chain Rule]\label{prop:ChainRule}
We have
\begin{align}\label{eq:KLDivChainRule}
    D(\ph_{n,k}\|p) = d(\ph_k\|p_k) + (1-\ph_k)D\left(\left(\frac{X_1}{n-X_k},\dotsc,\frac{X_{k-1}}{n-X_k}\right)\Bigg\|\left(\frac{p_1}{1-p_k},\dotsc,\frac{p_{k-1}}{1-p_k}\right)\right),
\end{align}
where the second term is considered to be 0 if $X_k = n$.
\end{proposition}
\begin{remark}\label{rem:ChainRuleZ}
Note that $(X_1,\dotsc,X_{k-1})\big| X_k \sim \Multinomial(n-X_k,p')$ where $p' = \left(\frac{p_1}{1-p_k},\dotsc,\frac{p_{k-1}}{1-p_k}\right)$. We can then multiply both sides of~\eqref{eq:KLDivChainRule} with $2n$ to restate \cref{prop:ChainRule} with slight abuse of notation as 
\begin{align}\label{eq:ChainRuleZ}
    Z_{n,k,p} = Z_{n,2,p_k} + Z_{n-X_k,k-1,p'}
\end{align}
(recall that $Z_{n,2,p_k}$ depends on $X_k$). 
\end{remark}
We now recall the result in~\cite{Agrawal2020} that shows sharp concentration for the case of binary alphabet size.
\begin{proposition}[Theorem III.2 in~\cite{Agrawal2020}]\label{prop:AgrawalMGF}
We have $\E\left[\exp\left\{Z_{n,2,p}/4\right\}\right] \le 2$, and so for all integers $m$, $\|Z_{n,2,p}\|_m \le 20m$.
\end{proposition}
\begin{remark}
The constant in the latter part of \cref{prop:AgrawalMGF} can be derived from~\cite[Theorem 3.14]{duchi2016lecture} or~\cite{Boucheron--Lugosi--Massart13}.
\end{remark}
As mentioned in~\cref{subsec:our_tech},
we need to handle the tails of the random variable  $\E[Z_{n-X_k,k-1,p'}|X_k]$ to perform induction.
This requires us to define and analyze the following terms, which appear in our proof.
\begin{definition}\label{def:fgphi} Define
\begin{align*}
    f_{k,p}(n) &:= \E[D(\ph_{n,k}\|p)] \text{ for } n \in \Nat\\
    g_{k,p}(n) &:= n\E[D(\ph_{n,k}\|p)] = nf_{k,p}(n) \text{ for } n \in \Nat\\
    \phi(x) &:= x \log(1/x) \text{ for } 0 < x \le 1, \text{ }\phi(0) = 0.
\end{align*}
\end{definition}
\begin{remark}\label{rem:GisExpectedZ}
From \cref{def:fgphi}, we see that $\E[Z_{n-X_k,k-1,p'}|X_k] = g_{k-1,p'}(n-X_k)$.
\end{remark}
As mentioned previously, the function $g$ plays an important role in the main proof---in particular, obtaining sharp upper and lower bounds on the function $g_{k,p}(n)$ becomes important. To this end, we use its relationship with Bernstein polynomials. The Bernstein polynomial approximation of a continuous function $h: [0,1] \to \Real$ is defined as \begin{align}\label{eq:BernsteinPolyDef}
    B_n(h,x) := \sum_{k=0}^n h(k/n) \binom{n}{k} x^k (1-x)^{n-k}.
\end{align}
Therefore, we see that for $N \sim \Bin(n,x)$, $B_n(h,x) = \Expt\left[h\left(\frac{N}{n}\right)\right]$.

We now recall the following relation between the expected value of relative entropy and Bernstein polynomials~\cite{braess2004bernstein,WuYang2020}.
\begin{restatable}{proposition}{propGFuncBernstein}
\label{prop:GFuncBernstein}
Let $\phi(x) := x \log (1/x)$ for $x \in (0,1], \phi(0) = 0$. We then have 
\[
    g_{k,p}(n) = \sum_{i=1}^k n\left(\phi(p_i) - B_n(\phi,p_i)\right).
\]
\end{restatable}
Finally, we also state the following properties of $f_{k,p}(n)$ which will be also useful further on~\cite{WuYang2020,Paninski2003}.
\begin{restatable}[Properties of $f_{k,p}(n)$]{proposition}{propFProperties} \label{prop:fProperties}
For any distribution $p \in \Delta^{k-1}$ and $n \in \Nat$, we have
\begin{enumerate}
    \item $f_{k,p}(n) \leq \frac{k-1}{n}$.
    \item $f_{k,p}(n) \le f_{k,p}(n+1)$ for $n \in \Nat$.
\end{enumerate}
\end{restatable}
 For completeness, we give the proofs of \cref{prop:GFuncBernstein,prop:fProperties} in \cref{app:proofprelim}.
\section{Proof of Main Results}\label{sec:MainThmPf}
In this section we prove the main results and defer the proof of intermediate lemmas to \cref{sec:LemmasForProof}.
\subsection{MGF and Centered Moments}\label{sec:MainCentered}
We begin by recalling the assertion of~\cref{thm:MainThm}:

Let $p \in \Delta^{k-1}_{\a}$ for an $\alpha >0$. Then for all $n \in \Nat$,
\begin{align}\label{eq:ThmStatementMain}
     \psi_{Z_{n,k,p}-\E[Z_{n,k,p}]}(t) \le Ck\log^4\left(\frac{k}{\alpha}\right)t^2, \text{ for all } |t| \le \frac{1}{c}
\end{align}
where $C,c$ are absolute positive constants. 
\begin{proof}[Proof of \cref{thm:MainThm}]
We proceed by induction on the alphabet size $k$. 

{\bf Base Case:} We proceed by first establishing that when $k=2$, then $Z_{n,2,p} -\E[Z_{n,2,p}] $ is a sub-Gamma random variable. This is proved in the following lemma, the proof of which is deferred to \cref{sec:LemmasForProof}.
\begin{restatable}[Concentration of centered binary divergence]{lemma}{lemCentBinary}\label{lem:BinaryDivConcCentered}
For absolute constants $C_2$ and $c_2>0$
\begin{align}\label{eq:BinDivMGF}
     \sup_{p \in [0,1]} \psi_{Z_{n,2,p}-\E[Z_{n,2,p}]}(t) \le C_2 t^2, \text{ for all } |t| \le \frac{1}{c_2}.
\end{align}
\end{restatable}
Using \cref{lem:BinaryDivConcCentered}, we have the following:
\[
    \sup_{p \in [0,1]} \psi_{Z_{n,2,p}-\E[Z_{n,2,p}]}(t) \le C_2 t^2 \le C t^2 
\]
for all $|t| \le c^{-1} \le c_2^{-1}$ by choosing absolute constants $C \ge C_2$ and $c \ge c_2$.

{\bf Induction step:} Assume the theorem holds for alphabet size $k-1$, i.e. for all $n \in \Nat$
\[
\sup_{p \in \Delta^{k-2}_{\a}}  \psi_{Z_{n,k-1,p}-\E[Z_{n,k-1,p}]}(t) \le C(k-1)\log^4\left(\frac{k-1}{\a}\right) t^2
\]
for all $|t| \le c^{-1}$.

{\bf Proof for alphabet size k:} We will now establish the result for alphabet size $k \ge 3$.%

 Since we have to show~\eqref{eq:ThmStatementMain} for all $n \in \Nat$, we will split $n$ into three ranges and use different arguments in each range.
 
\underline{\textit{Range 1:}} 
First, consider $n \le 4096(k-1) \log^2 (k-1)$. We use the bounded differences inequality and obtain the following lemma, the proof of which is relegated to \cref{app:LemmasForProof}. \begin{restatable}[Bounded differences inequality for $Z_{n,k,p}$]{lemma}{LemBddDiffsKL}\label{lem:BddDiffsKL}
For all $t \in \Real$, we have
\begin{align}
      \sup_{p \in \Delta^{k-1}_{\a}}  \psi_{Z_{n,k,p} - \E[Z_{n,k,p}]}(t)  \le 27 n \log^2\left(\frac{n}{\a}\right) t^2.
\end{align}
\end{restatable}
Therefore, when $n$ is in Range 1, we use \cref{lem:BddDiffsKL} to claim 
\begin{align}
 \psi_{Z_{n,k,p} - \E[Z_{n,k,p}]}(t)  &\le 27 n \log^2\left(\frac{n}{\a}\right) t^2  \nonumber \\
 &\le 27\times 4096 \times 9 (k-1) \log^4\left(\frac{k-1}{\a}\right) t^2 \label{eq:SecondEqRangeOne}\\
 &\le Ck \log^4\left(\frac{k}{\a}\right)t^2 \label{eq:InductionStepRangeOne}
\end{align}
for all $t \in \Real$ and large enough absolute constant $C$, where~\eqref{eq:SecondEqRangeOne} uses that $4096 k \log^2 k \le k^9$ for all $k \ge 3$. Therefore, when $n \le 4096(k-1) \log^2 (k-1)$, we have proved the induction step.

\underline{\textit{Above range 1:}} Now let $n > 4096(k-1) \log^2 (k-1)$. Assume without loss of generality that $p = (p_1,\dotsc,p_k) \in \Delta^{k-1}_{\a}$ with $p_1 \ge \dotsc \ge p_k \ge \a > 0$, and let $|t| \le c^{-1}$. Define 
\begin{align}\label{eq:pPrimeDefn}
p' := \left(\frac{p_1}{1-p_k},\dotsc,\frac{p_{k-1}}{1-p_k}\right) \in \Delta^{k-2}_{\alpha}
\end{align}
where we observe that $\frac{p_i}{1-p_k} \ge p_i \ge \a$ for $1 \le i \le k-1$. Let $X_1,\dotsc,X_k \sim \Multinomial(n,p)$. From \cref{prop:ChainRule}, we have 
\begin{align*}
Z_{n,k,p} = Z_{n,2,p_{k}} + Z_{n-X_k,k-1,p'}
\end{align*}
where the second term is 0 if $X_k = n$. Then,
\begin{align}
    &Z_{n,k,p}-\E[Z_{n,k,p}]\nonumber\\
    &= Z_{n,2,p_k}-\E[Z_{n,2,p_k}] + Z_{n-X_k,k-1,p'} - \E[Z_{n-X_k,k-1,p'}] \nonumber \\
    &= Z_{n,2,p_k}-\E[Z_{n,2,p_k}] + Z_{n-X_k,k-1,p'} -\E[Z_{n-X_k,k-1,p'}|X_k] + \E[Z_{n-X_k,k-1,p'}|X_k] - \E[Z_{n-X_k,k-1,p'}]  \nonumber 
\end{align}
and by noting that $\E[Z_{n-X_k,k-1,p'}|X_k] = 2g_{k-1,p'}(n-X_k)$, we have 
\begin{align}
   &Z_{n,k,p}-\E[Z_{n,k,p}]\nonumber\\
   &\quad=  Z_{n,2,p_k}-\E[Z_{n,2,p_k}] + Z_{n-X_k,k-1,p'} -\E[Z_{n-X_k,k-1,p'}|X_k] +  2g_{k-1,p'}(n-X_k) - \E[2g_{k-1,p'}(n-X_k)]\label{eq:ThreeTerms}.
\end{align}
Now, by the tower property of expectation, we have 
\begin{align}\label{eq:ConditionOnXk}
 \E\left[\exp \left\{t(Z_{n,k,p}-\E[Z_{n,k,p}])\right\}\right] =
    \E\left[\E\left[\exp \left\{t(Z_{n,k,p}-\E[Z_{n,k,p}])\right\}\Big|X_k\right]\right]
\end{align} 
and noticing that $Z_{n,2,p_k}-\E[Z_{n,2,p_k}]$ and $g_{k-1,p'}(n-X_k) - \E[g_{k-1,p'}(n-X_k)]$ are both functions of $X_k$, we have 
\begin{align}
    \E&\left[\exp \left\{t(Z_{n,k,p}-\E[Z_{n,k,p}])\right\}\right] \nonumber\\
    &\qquad\quad= \E\Bigg[\exp\left\{t(Z_{n,2,p_k}-\E[Z_{n,2,p_k}])\right\} \cdot  \E\left[ \exp\left\{t(Z_{n-X_k,k-1,p'}-\E[Z_{n-X_k,k-1,p'}|X_k])\right\}\Big|X_k\right] \nonumber\\
    &\qquad\qquad\qquad\qquad\qquad\qquad\qquad\qquad\quad \cdot\exp\left\{2t\left(g_{k-1,p'}(n-X_k) - \E[g_{k-1,p'}(n-X_k)]\right)\right\}\Bigg] \label{eq:ThreeMGFTerms}
\end{align}
and further using the H\"{o}lder inequality 
\begin{align}
     &\left(\E\left[\exp \left\{t(Z_{n,k,p}-\E[Z_{n,k,p}])\right\}\right]\right)^{3} \nonumber\\ 
     &\quad\le \E\left[\exp \left\{3t(Z_{n,2,p_k}-\E[Z_{n,2,p_k}])\right\}\right]  \E\left[\left(\E\left[ \exp\left\{t(Z_{n-X_k,k-1,p'}-\E[Z_{n-X_k,k-1,p'}|X_k])\right\}\Big|X_k\right]\right)^{3}\right]  \nonumber\\
     &\qquad\qquad\qquad\qquad\qquad\qquad\qquad\qquad\qquad\E\left[\exp\left\{6t\left(g_{k-1,p'}(n-X_k) - \E[g_{k-1,p'}(n-X_k)]\right)\right\}\right]. \label{eq:takelogarithm}
\end{align}
Now, by the base case we have for $3|t| \le 3c^{-1} \le c_2^{-1}$, 
\begin{align}\label{eq:binaryMGFbd}
    \E\left[\exp \left\{3t(Z_{n,2,p_k}-\E[Z_{n,2,p_k}])\right\}\right] \le \exp\{9C_2 t^2\}
\end{align}
and by the induction hypothesis for $|t| \le c^{-1}$
\begin{align}\label{eq:MiddleTermInduction}
    \left(\E\left[ \exp\left\{t(Z_{n-X_k,k-1,p'}-\E[Z_{n-X_k,k-1,p'}|X_k])\right\}\Big|X_k\right]\right)^{3} \le \exp\left\{3C(k-1)\log^4 \left(\frac{k-1}{\a}\right)t^2\right\}.
\end{align}
In particular,~\eqref{eq:MiddleTermInduction} holds almost surely for any $X_k$ since the induction hypothesis holds for \emph{all } natural numbers (and the left hand side is zero when $X_k = n$).
Now, using~\eqref{eq:binaryMGFbd} and~\eqref{eq:MiddleTermInduction} in~\eqref{eq:takelogarithm} we get
\begin{align}\label{eq:logMGFThreeTerms}
    \psi_{Z_{n,k,p}-\E[Z_{n,k,p}]}(t) \le 3C_2 t^2 + 3C(k-1) \log^4 \left(\frac{k-1}{\a}\right) t^2 + \frac{1}{3}\psi_{g_{k-1,p'}(n-X_k) - \E[g_{k-1,p'}(n-X_k)]}(6t).
\end{align}
\underline{\textit{Range 2:}} We will now consider the range $4096(k-1) \log^2(k-1) \le n < \frac{128 (k-1)^2}{\a^2 \sqrt{\a}}$\footnote{Note that $\a = \min\{p_1,\dotsc,p_k\} \le 1/k$, so  $\frac{128(k-1)^2}{\a^2\sqrt{\a}} \ge 128(k-1)^2 k^2 \sqrt{k} \ge 4096(k-1)\log^2(k-1)$ for all $k \ge 3$.}. 
We will use the following result saying that when $n \geq 4096(k-1)\log^2(k-1)$, then $g_{k-1,p'}$ satisfies the required concentration around its mean.
\begin{restatable}
[Concentration of $g_{k,p}(Y)$, Range 2]{lemma}{lemGsubGammaRangeOne}
\label{lem:GsubGammaRangeOne}
Let $Y \sim \Bin(n,r)$ where $r \ge 1-\frac{1}{k} \ge \frac{1}{2}$ and $n \ge 4096k \log^2k$. Then we have for all $p \in \Delta^{k-1}$
\begin{align*}
    \psi_{g_{k,p}(Y) - \E[g_{k,p}(Y)]}(t) \le (C_g\log^2 n) t^2, \text{ for all } |t| \le \frac{1}{c_g}
\end{align*}
where $C_g \le 1536 \times 2048, c_g \le 288$.
\end{restatable}
The proof of \cref{lem:GsubGammaRangeOne} is deferred to \cref{sec:LemmasForProof}.
As $n < \frac{128 (k-1)^2}{\a^2 \sqrt{\a}}$, \cref{lem:GsubGammaRangeOne} implies the following bound for $g_{k-1,p'}$:
\begin{align}
    \psi_{g_{k-1,p'}(n-X_k) - \E[g_{k-1,p'}(n-X_k)]}(6t) \le 36C_g \log^2\left(\frac{128 (k-1)^2}{\a^2 \sqrt{\a}}\right) t^2 \le 288C_g \log^2\left(\frac{k-1}{\a}\right) t^2
\end{align} 
and therefore from~\eqref{eq:logMGFThreeTerms} we have whenever $4096 (k-1) \log^2 (k-1) < n \le \frac{128(k-1)^2}{\a^2 \sqrt{\a}}$ that 
\begin{align}
     \psi_{Z_{n,k,p}-\E[Z_{n,k,p}]}(t) &\le 3C_2 t^2 + C(k-1)\log^4 \left(\frac{k-1}{\a}\right)t^2 + 288C_g \log^2\left(\frac{k-1}{\a}\right) t^2 \nonumber \\
     &\le Ck \log^4 \left(\frac{k}{\a}\right)\label{eq:InductionStepRangeTwo}
\end{align}
for $|t| \le c^{-1}$, where~\eqref{eq:InductionStepRangeTwo} holds for any absolute constant $C \ge 3C_2 + 288C_g$. Therefore the induction hypothesis is proved for this range of $n$.

\underline{\textit{Range 3:}} We finally consider the case when  $n \ge \frac{128(k-1)^2}{\a^2 \sqrt{\a}}$.
The variance term in \cref{lem:GsubGammaRangeOne} scales with $\log n$, which diverges as $n\to \infty$. We now state the following the lemma that uses a different argument to get rid of this dependence on $\log n$ when $n \gtrsim k^2/\a^{5/2}$.
\begin{restatable}
[Concentration of $g_{k,p}(Y)$, Range 3]{lemma}{lemGConcLargeN}  \label{lem:GConcLargeN}
There exists a constant $C_g' > 0$ such that if $n \ge \frac{48 k^2}{\a^2 \sqrt{\a}}$ and $Y \sim \Bin(n,r)$ with $r \ge 1-\frac{1}{k} \ge \frac{1}{2}$, then for all $t \in \Real$ and $p \in \Delta^{k-1}_{\a}$
\begin{align}
    \psi_{g_{k,p}(Y) - \E[g_{k,p}(Y)]}(t) \le C_g' t^2. 
\end{align}
\end{restatable}
The proof of \cref{lem:GConcLargeN} is provided in \cref{sec:LemmasForProof}.
Using \cref{lem:GConcLargeN}, we have the following bound on $g_{k-1,p'}$:
\[
 \psi_{g_{k-1,p'}(n-X_k) - \E[g_{k-1,p'}(n-X_k)]}(6t) \le 36C_g' t^2
\]
for all $6|t| \le   6c^{-1} \le c_g'^{-1}$ (by choosing $c \ge 6c_g'$). And therefore from~\eqref{eq:logMGFThreeTerms} we have whenever $n \ge \frac{128 (k-1)^2}{\a^2 \sqrt{\a}}$ that 
\begin{align}
     \psi_{Z_{n,k,p}-\E[Z_{n,k,p}]}(t) &\le 3C_2 t^2 + C(k-1)\log^4 \left(\frac{k-1}{\a}\right)t^2 + 36C_g't^2 \nonumber\\
     &\le Ck \log^4\left(\frac{k}{\a}\right) t^2 \label{eq:InductionStepRangeThree}
\end{align}
where we choose an absolute constant $C$ large enough to satisfy $C \ge 3C_2 + 36C_g'$ establishing the theorem in this case as well. 
\begin{remark}[Bound on $C$ and $c$]\label{rem:MainConstBds}
In light of the proof, we can see that taking $c = \max\{3c_2, 6c_g, 6c_g'\} \lesssim 1$ and $C = \max\{3C_2+288C_g, 3C_2+36C_g'\} \lesssim 1$ suffices. 
\end{remark}
\end{proof}
Once \cref{thm:MainThm} is established, \cref{cor:ConfIntCentered,cor:MomentBdsCentered} follow from \cref{prop:SubGTailsMoments}.
\subsection{Raw Moments}\label{subsec:MainRaw}
Finally, we prove that the raw moments of $Z_{n,k,p}$ matches the asymptotic limit (up to constants) with finite samples.
\propRawMomentBds*
\begin{proof}
We begin with the following restatement of the result of \cite[Theorem I.2]{Agrawal2020}.
\begin{restatable}[\cite{Agrawal2020}]{claim}{resClaimAgrRawConf}
\label{claim:AgrRawConf}
There is an absolute constant $C$ such that for all $\delta \geq 0$, with probability at least $1 - \delta$, 
$
Z_{n,k,p} \leq C\left(k + \log\left( \frac{1}{\delta} \right)  \right).
$\end{restatable}
For completeness, we provide the proof of \cref{claim:AgrRawConf} in  \cref{app:LemmasForProof}.
Using \cref{claim:AgrRawConf}, we can show our bound as follows.
Let $Y = \left(Z_{n,k,p} -  Ck\right)_+$.
By \cref{claim:AgrRawConf}, we have that with probability $1 -\delta$, $|Y| \leq C\log(1/\delta)$.
Thus applying \cref{prop:SubGTailsMoments2}, we have that $Y - \E[Y] \in \Gamma(864C^2,144C)$, and furthermore applying \cref{prop:SubGTailsMoments}, we have that $\|Y -\E[Y]\|_q \lesssim q $.
Using the fact that $0 \leq Y \leq Z_{n,k,p}$, we have that $0 \leq \E[Y] \leq \E[Z_{n,k,p}]  = 2g_{k,p}(n)\leq 2k$ (\cref{prop:fProperties}). Finally, using the triangle inequality for $L_q$ norms, we have the following:
\begin{align*}
    \|Z_{n,k,p}\|_m &= \|Y  + (Ck - Z_{n,k,p})_+ \|_m \\
    &\leq  \|Y\|_m + \|(Ck - Z_{n,k,p})_+ \|_m \\
    &\leq   \|Y -\E[Y]\|_m + \|\E[Y]\|_m +  \|Ck\|_m \\ &\lesssim m + k.
\end{align*}
This completes the proof.
\end{proof}
\section{Proofs of Intermediate Lemmas}\label{sec:LemmasForProof}
In this Section, we establish the concentration of several random variables, which were required to establish our main result. To begin with, we have the following result about the centered binary KL divergence that follows from a combination of the main result of~\cite{Agrawal2020} and our centering lemma. 
\lemCentBinary*
\begin{proof}
From \cref{prop:AgrawalMGF}, we have $\E\left[Z_{n,2,p}\right]^{2q} \le {2q}^{2q}(20)^{2q} \le 2q! (20)^{2q}$ for any $q \ge 1$.
Therefore we can invoke \cref{prop:SubGTailsMoments2} with $A = 0, B = 20$ to establish that $Z_{n,2,p} - \E[Z_{n,2,p}] \in \Gamma(\nu,c)$ with $\nu \le 14400, c \le 120$. The lemma then follows with $C_2 \le 14400, c_2 \le 240$.
\end{proof}

\subsection{Concentration of \texorpdfstring{$g$}{}: Range 2}\label{subsec:ConcGSmallN}
Recall that in~\cref{sec:MainThmPf}, Range 2 of $n$ corresponded to $4096(k-1) \log^2(k-1) \le n < \frac{128 (k-1)^2}{\a^2 \sqrt{\a}}$. We aim to establish concentration properties of $g_{k,p}(Y)$ where $Y \sim \Bin(n,r)$ for a particular $r \in [0,1]$ and $n$ within this range. As a first step towards this, we establish a bound on the discrete gradient of $f_{k,p}(\cdot)$ in the next lemma (proof relegated to \cref{app:LemmasForProof}). 
 
\begin{restatable}[Discrete gradient of $f_{k,p}(\cdot)$]{lemma}{lemDiscGradF} \label{lem:DiscGradF} 
\label{LemFDifference}
For any $n \ge 4$, $k \in \Nat$ and $p \in \Delta^{k-1}$ we have that 
\begin{align*}
0 \leq f_{k,p}(n) - f_{k,p}(n+1) &\le \frac{\sqrt{k}\log n}{n^{3/2}}  + \frac{8k\log^2 n}{n^{2}}.
\end{align*}
\end{restatable}
Using the discrete gradient of $f_{k,p}(n)$, we can establish the following result. 
\lemGsubGammaRangeOne*

\begin{proof} 
Extending the definitions of $f_{k,p}(\cdot), g_{k,p}(\cdot)$ to all $x \ge 0$, we define 
\begin{align}
    f_{k,p}(0) &= 0,  &f_{k,p}(x) &= f_{k,p}(\lf x \rf), \nonumber\\
    g_{k,p}(0) &= 0,   &g_{k,p}(x) &= g_{k,p}(\lf x \rf) \label{eq:Extendedfgdefns}.
\end{align}
We point out that~\eqref{eq:Extendedfgdefns} ensures that $f_{k,p}(x) = \lf x\rf g_{k,p}(\lf x \rf)$ for all $x \ge 0$.

Now, consider for $Y \sim \Bin(n,r)$
\begin{align}
\left|g_{k,p}(Y) - g_{k,p}(nr)\right| &= \left|Y f_{k,p}(Y) - \lf nr \rf f_{k,p} (\lf nr \rf )\right| \nonumber \\
&= \left|Y f_{k,p}(Y) - \lf nr \rf f_{k,p}(Y) + \lf nr \rf f_{k,p}(Y) - \lf nr \rf f_{k,p} (\lf nr \rf )\right| \nonumber\\
&\le |Y-\lf nr \rf|f_{k,p}(Y) + \lf nr \rf |f_{k,p}(Y) - f_{k,p}(\lf nr \rf)| \nonumber\\
&\le (|Y - nr|+1)f_{k,p}(Y) + n |f_{k,p}(Y) - f_{k,p}(\lf nr \rf)|. \label{eq:DiscGradGF}
\end{align}
Now, we have by the Chernoff bound (see for example~\cite[Exercise 2.12]{Boucheron--Lugosi--Massart13}) 
\begin{align}\label{eq:chernoffBdDelta}
    |Y - nr| \le \sqrt{2nr(1-r)\log\frac{2}{\d}} \le \sqrt{\frac{2n}{k}\log \frac{2}{\d}}
\end{align}
with probability $\ge 1-\d$.  

Consider now $0 < \d \le 1$ such that $\log \frac{2}{\d} \le k$. Then, $\sqrt{\frac{2n}{k}\log \frac{2}{\d}} \le \sqrt{2n}$ and consequently, with probability greater than $1-\d$, we have 
\[
Y \ge nr - \sqrt{2n} \ge \frac{n}{2}-\sqrt{2n} \ge \frac{n}{4}.
\]
where the final inequality holds for any $n \ge 32$. Now, if $Y \ge \frac{n}{4} \ge 1$, we have 
\begin{align}
    n|f_{k,p}(Y) - f_{k,p}(\lf nr \rf)| &= n \sum_{i=\max\{Y,\lf nr \rf\}}^{\min\{Y,\lf nr \rf\}} (f_{k,p}(i) - f_{k,p}(i+1)) \label{eq:FDecreasing} \\
    &\le n \sum_{i=\max\{Y,\lf nr \rf \}}^{\min\{Y,\lf nr \rf\}} \left( \frac{\sqrt{k} \log i}{i^{3/2}} + \frac{8k \log^2 i}{i^2}   \right) \label{eq:UseDiscGrad} \\
    &\le (|Y-\lf nr \rf |+1) \left( \frac{8\sqrt{k} \log n}{n^{1/2}} + \frac{128k \log^2 n}{n}   \right) \label{eq:FunctionDecrease} \\
    &\le (|Y- nr |+2)  \frac{16\sqrt{k} \log n}{\sqrt{n}} \label{eq:SampleCompCondition}
\end{align}
where~\eqref{eq:FDecreasing} follows since $f_{k,p}(i) \ge f_{k,p}(i+1)$ for $i \in \Nat$ (\cref{prop:ChainRule}) and a telescoping sum,~\eqref{eq:UseDiscGrad} follows from \cref{lem:DiscGradF}, ~\eqref{eq:FunctionDecrease} follows by noting $i \mapsto \frac{\sqrt{k} \log i}{i^{3/2}} + \frac{8k \log^2 i}{i^2}$ is decreasing for $i \ge 3$ and $\min(Y, \lf r\rf) \ge n/4 \ge 3$, and~\eqref{eq:SampleCompCondition} follows since $\frac{n}{256\log^2 n} \ge k$ ensuring that the first term in~\eqref{eq:SampleCompCondition} is larger. 
Therefore, in the event that $|Y - nr| \le \sqrt{\frac{2n}{k}\log \frac{2}{\d}}$ for $\d$ such that $\log \frac{2}{\d} \le k$, we have from~\eqref{eq:DiscGradGF} 
\begin{align}
    \left|g_{k,p}(Y) - g_{k}(nr)\right| &\le (|Y-nr|+1)f_{k,p}(Y) + 16(|Y-nr|+2)\frac{\sqrt{k}\log n}{n} \nonumber\\
    &\le (|Y-nr|+2)\left(\frac{4k}{n} + \sqrt{\frac{256k \log^2 n}{n}}\right) \label{eq:UseBdOnY} \\
    &\le 32\left(\sqrt{\frac{2n}{k}\log \frac{2}{\d}} + 2\right)\sqrt{\frac{k \log^2 n}{n}} \label{eq:Samplecomplexitytwo} \\
    &\le \sqrt{2048 \log^2 n \left(\log\frac{1}{\d}+1\right)} + 1 \label{eq:samplecomplexitythree} \\
    &\le \sqrt{2048 \log^2 n\log\frac{1}{\d}} + 32\sqrt{2} \log n + 1  \label{eq:sqrtaplusb}
\end{align}
where~\eqref{eq:UseBdOnY} uses that in the given event $Y \ge n/4$ (and thus $f_{k,p}(Y) \le f_{k,p}(n/4)$),~\eqref{eq:Samplecomplexitytwo} and~\eqref{eq:samplecomplexitythree} use that $n$ is large enough to ensure $\frac{n}{4096\log^2 n} \ge k$, and~\eqref{eq:sqrtaplusb} uses the fact that for $a,b > 0$, $ \sqrt{a+b} \le \sqrt{a}+\sqrt{b}$.

We have therefore established that for all $0 < \d \le 1$ such $\log \frac{2}{\d} \le k$, we have 
\begin{align}\label{eq:GBoundPartOne}
    |g_{k,p}(Y) - g_{k,p}(nr)| \le \sqrt{2048 \log^2 n\log\frac{1}{\d}} + 32\sqrt{2} \log n + 1 \text{ with probability } \ge 1-\d.
\end{align}

Now consider all $\d$ such that $\log \frac{2}{\d} \ge k$. In this case, since $g_{k,p}(Y) \in [0,k]$, we have 
\begin{align}\label{eq:GBoundPartTwo}
    |g_{k,p}(Y) - g_{k,p}(nr)| \le \log \frac{2}{\d} \le 1 + \log \frac{1}{\d}. 
\end{align}
Putting together~\eqref{eq:GBoundPartOne} and~\eqref{eq:GBoundPartTwo} we have established that for all $0 < \d \le 1$,
\begin{align}\label{eq:finalGbound}
     |g_{k,p}(Y) - g_{k,p}(nr)| \le  \sqrt{2048 \log^2 n} + 1  + \sqrt{2048 \log^2 n\log\frac{1}{\d}} + \log \frac{1}{\d} \text{    w.p. } \ge 1-\d
\end{align}
and therefore $g_{k,p}(Y) - g_{k,p}(nr)$ satisfies~\eqref{eq:UpperBoundToMGF} in \cref{prop:SubGTailsMoments} with $A = 1024 \log^2 n$ and $B = 1$, implying that $g_{k,p}(Y) - \E[g_{k,p}(Y)] \in \Gamma(1536A + 864B^2, 144B)$, leading to the assertion.

\end{proof}

\subsection{Concentration of \texorpdfstring{$g$}{}: Range 3}\label{subsec:concGlargeN}

When $n \ge \frac{128 (k-1)^2}{\a^2 \sqrt{\a}}$ (i.e. $n$ is in Range 3 corresponding to the notation in~\cref{sec:MainThmPf}), we can establish a sharper version of \cref{lem:GsubGammaRangeOne}. Intuitively, since $g_{k,p}(n) \to \frac{k-1}{2}$, we expect that for a large enough $n$, $\left|g_{k,p}(n) - \frac{k-1}{2}\right|$ is small. The next proposition quantifies this intuition. 
\begin{proposition}[Bound on $g_{k,p}(n)$] \label{prop:Gbound}
Let $n \ge \frac{16 k^2}{\a^2 \sqrt{\a}}$ and $p \in \Delta^{k-1}_{\a}$. Then, $\left|g_{k,p}(n) - \frac{k-1}{2}\right| \le 1$.
\end{proposition}
\begin{proof}
Define for $x \in \{0\} \cup \left[\frac{1}{n}, 1\right]$
\begin{align}\label{eq:RemainderDefn}
    R_i(x) = x \log \frac{x}{p_i} - (x-p_i) - \frac{(x-p_i)^2}{2p_i}.
\end{align}
We then have $R_i(0) = \frac{p_i}{2}$.
When $\frac{1}{n} \le x \le \frac{p_i}{2}$, we use the Taylor theorem to assert 
\begin{align}\label{eq:TaylorThmPhi}
    \phi(x) = \phi(p_i) + \phi'(p_i)(x-p_i) + \frac{\phi''(p_i)}{2}(x-p_i)^2 + \frac{\phi'''(c)}{6}(x-p_i)^3
\end{align}
where $c \in [x,p_i]$. Recalling that $\phi(p_i) = p_i \log \frac{1}{p_i}, \phi'(p_i) = \log \frac{1}{p_i}-1, \phi''(p_i) =  -1/p_i$ and $\phi'''(c) = 1/c^2$ we can rewrite~\eqref{eq:TaylorThmPhi} as 
\begin{align}
    x \log \frac{1}{x} &= p_i \log \frac{1}{p_i} + (x-p_i)\log \frac{1}{p_i} -(x-p_i) - \frac{(x-p_i)^2}{2p_i} + \frac{1}{6c^2}(x-p_i)^3 \nonumber\\
    &=  x \log \frac{1}{p_i} -(x-p_i) -\frac{(x-p_i)^2}{2p_i} + \frac{1}{6c^2}(x-p_i)^3 \nonumber
\end{align}
and rearranging we get 
\[
R_i(x) = x\log \frac{x}{p_i}-(x-p_i) -\frac{(x-p_i)^2}{2p_i} =  \frac{-1}{6c^2}(x-p_i)^3 
\]
for a $c \in [x,p_i]$. Now, if $1/n \le x \le p_i/2$, we have $|R_i(x)| \le \frac{n^2}{6}|x-p_i|^3 \le \frac{n^2}{6}$. Furthermore, if $x \ge p_i/2$, we have $c \ge p_i/2$,%
and therefore $|R_i(x)| \le \frac{2}{3p_i^2}|x-p_i|^3$. We have therefore established that 
 \[
    |R_i(x)| \le \Bigg\{ \begin{array}{lr}
        p_i/2 & \text{for } x = 0\\
        n^2/6 & \text{for } 1/n \leq x\leq p_i/2\\
        (2/3p_i^2)|x-p_i|^3 & \text{for } x > p_i/2
        \end{array}.
  \]
Now, by definition of $R_i(x)$, we have 
\begin{align*}
    \sum_{i=1}^k \ph_i \log \frac{\ph_i}{p_i} - (\ph_i - p_i) - \frac{(\ph_i-p_i)^2}{2p_i} &= \sum_{i=1}^k R_i(\ph_i).
\end{align*}
Taking expectation on both sides yields%
\[
\frac{g_{k,p}(n)}{n} - \frac{k-1}{2n} = \sum_{i=1}^k \E[R_i(\ph_i)]
\]
and by the triangle inequality, 
\begin{align}\label{eq:TriangleIneqRi}
\left|\frac{g_{k,p}(n)}{n} - \frac{k-1}{2n}\right| \le \sum_{i=1}^k \E|R_i(\ph_i)|.
\end{align}
Now, using the upper bounds on $R_i(x)$ we have 
\begin{align}
    \E[|R_i(\ph_i)|] &\le \frac{p_i}{2}\E\left[\indic{\ph_i = 0}\right] + \frac{n^2}{6}\E\left[\indic{\frac{1}{n} \le \ph_i \le \frac{p_i}{2}}\right] + \frac{2}{3p_i^2}\E\left[|p_i - \ph_i^3| \indic{\ph_i \ge \frac{p_i}{2}}\right] \nonumber\\
    &\le \frac{n^2}{6}\Prob(\ph_i \le p_i/2) + \frac{2}{3p_i^2} \E|\ph_i - p_i|^3. \label{eq:UseChernoffBd}
\end{align}
Now we have by the Chernoff Bound~\cite[Theorem 4.5]{Mitzenmacher--Upfal2017}
\begin{align}\label{eq:Chernoffpi}
\Prob(\ph_i \le p_i/2) = \Prob(\ph_i \le p_i - p_i/2) \le \exp\left(-np_i/8\right) \le \exp\left(-n\a/8\right)
\end{align}
and by moments of the binomial distribution~\cite{Knoblauch2008}
\begin{align}\label{eq:BinomialThirdCentralMoment}
    \E|\ph_i - p_i|^3 \le  \left(\E|\ph_i - p_i|^4\right)^{3/4} \le \left(\frac{3p_i(1-p_i)}{n^2}\right)^{3/4} \le \frac{3p_i^{3/4}}{n\sqrt{n}}.
\end{align}
Using~\eqref{eq:Chernoffpi} and~\eqref{eq:BinomialThirdCentralMoment} in~\eqref{eq:UseChernoffBd} yields 
\begin{align}\label{eq:RiExptBd}
 \E[|R_i(\ph_i)|] \le \frac{n^2}{6}\exp(-n\a/8) + \frac{2}{p_i^{5/4}}\frac{1}{n\sqrt{n}} \le \frac{n^2}{6}\exp(-n\a/8) + \frac{2}{\a^{5/4}}\frac{1}{n\sqrt{n}}
\end{align}
and since for $n \ge \frac{16k^2}{\a \sqrt{\a}}$ we have $\frac{n^2}{6}\exp(-n\a/8) \le \frac{2}{\a^{5/4}}\frac{1}{n\sqrt{n}}$, we have 
\[
\E[|R_i(\ph_i)|] \le \frac{4}{\a^{5/4}}\frac{1}{n\sqrt{n}}.
\]
Using this in~\eqref{eq:TriangleIneqRi} and multiplying by n, we get 
\[
\left|g_{k,p}(n) - \frac{k-1}{2}\right| \le \frac{4k}{\a^{5/4}\sqrt{n}} \le 1
\]
for $n \ge \frac{16k^2}{\a^2 \sqrt{\a}}$.
\end{proof}

Now, using \cref{prop:Gbound}, we see that $g_{k,p}(n)$ is fairly constant for large enough $n$. Therefore, using the fact that bounded random variables are sub-Gaussian, we can establish the following concentration inequality. 
\lemGConcLargeN*
\begin{proof}
We claim that 
\begin{align}\label{eq:gConcLargeN}
    \Prob\left(\left|g_{k,p}(Y) - \frac{k-1}{2}\right| > t\right) \le 2e^{-t^2/4}
\end{align}
for all $t \ge 0$. To establish this we consider three cases.
First, whenever $t \le 1, 2e^{-t^2/4} > 1$ and therefore~\eqref{eq:gConcLargeN} is trivially true. 

When $1 < t \le k$, we know that if $Y \ge \frac{n}{3} = \frac{16 k^2}{\a^2 \sqrt{\a}}$, then we have $|g_{k,p}(Y) - \frac{k-1}{2}| \le 1$ from \cref{prop:Gbound}. Thus, for any $t \ge 1$ we can bound
\[
 \Prob\left(\left|g_{k,p}(Y) - \frac{k-1}{2}\right| > t\right) \le \Prob\left(Y \le \frac{n}{3}\right) \le e^{-n/36}
\]
from the Chernoff bound and that $r \ge 1-\frac{1}{k} \ge \frac{1}{2}$. Moreover, $e^{-n/36} \le e^{-t^2/4}$ since $t \le k \le \frac{\sqrt{n}}{3}$, and therefore~\eqref{eq:gConcLargeN} holds.  

Finally, whenever $t \ge k$, since $0 \le g_{k,p}(Y) \le k$, we have $\left|g_{k,p}(Y) - \frac{k-1}{2}\right| \le k$  with probability 1. Therefore in this case $\Prob\left(\left|g_{k,p}(Y) - \frac{k-1}{2}\right| > t\right) = 0 \le 2e^{-t^2/4}$ and~\eqref{eq:gConcLargeN} holds in this case too. The result then follows by using the centering lemma~\cite[Lemma 2.6.8]{vershynin2018high}. 
\end{proof}

\section{Concluding Remarks}\label{sec:discussion}

In this work we studied concentration inequalities for the centered relative entropy between the empirical distribution and the true (multinomial) distribution of alphabet size $k$ and minimum probability $\a > 0$. We showed that the centered relative entropy is sub-Gamma with variance parameter $\nu \lesssim k \log^4\left(\frac{k}{\a}\right)$ and shape parameter $c \lesssim 1$. This matches the $\chi^2_{k-1}$ asymptotics up to the $\log^4\left(\frac{k}{\a}\right)$ factor. Two remaining problems thus readily present themselves: firstly, removing any dependence on $\a$ in the logarithmic term in the variance factor (as the convergence in distribution to $\chi^2_{k-1}$ holds for arbitrary $\a > 0$); and secondly the more stringent question of establishing that $\nu \lesssim k$ (i.e. removing the extraneous logarithmic factor altogether). 
 \appendix
\section{Proofs Omitted From \texorpdfstring{\cref{sec:prelim}}{}}
\label{app:proofprelim}

\propSubGTailsMomentsII*
\begin{proof}
We will follow the same strategy as \cite[Theorem 2.3]{Boucheron--Lugosi--Massart13} on $X-\E[X]$.
First of all, note that by the Jensen inequality and the condition~\eqref{eq:MomentToMGF}
\begin{align}\label{eq:ExpectationXSq}
|\E[X]| \leq \E|X| \leq \sqrt{\E X^2} \leq \sqrt{A + 2B^2}.
\end{align}

Consider now the first assertion. We have 
\begin{align}\label{eq:RawMomentXBd}
    (\E[X])^{2q} \leq (A + 2B^2)^{q} \leq 2^{q-1}(A^{q} + 2^qB^{2q}) \le (2A)^{q} + (2B)^{2q},
\end{align}
where the first inequality follows from~\eqref{eq:ExpectationXSq} and the second follows by the convexity of $x \mapsto x^q$ since $q \ge 1$.
Next, we observe
\begin{align}
    \E[(X-\E[X])^{2q}] &\leq 2^{2q - 1}\left(\E[X^{2q}] + (\E[X])^{2q} \right) \label{eq:jensenPower} \\
                &\le 2^{2q} \left( q!A^q +  (2q)!B^{2q} + (2A)^{q} + (2B)^{2q}  \right) \label{eq:SubstituteRawMomentBd}\\
                &\leq q!(6A)^q + (2q)!(3B)^{2q}   \label{eq:factorialsCalc}
\end{align}
where~\eqref{eq:jensenPower} follows by the Jensen inequality,~\eqref{eq:SubstituteRawMomentBd} follows by~\eqref{eq:RawMomentXBd} and~\eqref{eq:factorialsCalc} follows since $q!(4A)^q + (2A)^q \le q!(6A)^q, (2q)!B^{2q} + (2B)^{2q} \le (2q)!(3B)^{2q}$. We can now see that $X - \E[X]$ is a centered random variable that satisfies the conditions of \cite[Theorem 2.3]{Boucheron--Lugosi--Massart13} with $A' = 6A$, $B' = 3B$. Thus $Z \in \Gamma(4(A' + B'^2),2B') = \Gamma(24A + 36B^2,6B)$.

We will now focus on the second condition and calculate the moments of $|X|$. From~\eqref{eq:UpperBoundToMGF}, we see that $\Prob( (|X| - \sqrt{2A} -  B)_+ \geq \sqrt{2A t} + Bt) \leq \exp(-t)$ for all $t > 0$. Using the same analysis as in the proof of \cite[Theorem 2.3]{Boucheron--Lugosi--Massart13}, we get that for all integers $q\geq 1$, 
\begin{align*}
    \E \left[\left(|X| - \sqrt{2A} -  B\right)_{+}\right]^{2q} \leq q!(8A)^{q} + (2q)!(4B)^{2q}.
\end{align*}
Next since $|X| \le \left(|X| - \sqrt{2A} -  B\right)_{+} + \sqrt{2A}+B$ using similar arguments as in~\eqref{eq:jensenPower}---\eqref{eq:factorialsCalc} we have,
\begin{align*}
    \E X^{2q} &\leq 2^{2q} \E \left[\left(|X| - \sqrt{2A} -  B\right)_{+}\right]^{2q} + 2^{2q} (\sqrt{2A} + B)^{2q}  \\
            &\leq q! (32A)^{q} + (2q)! (8B)^{2q} + 4^{2q} (2A)^{q} + 4^{2q}B^{2q} \\
            &\leq q! (64A)^q + (2q)!(24B)^{2q}.
\end{align*}
We have now reduced the problem to the first case and by the result therein we have $X - \E [X] \in \Gamma(1536 A + 864B^2, 144B)$.
\end{proof}
\propGFuncBernstein*
\begin{proof}
\begin{align}
    g_{k,p}(n) &= n\E[D(\ph_{n,k}\|p)] \nonumber\\
    &= n\E\left[\sum_{i=1}^k\ph_i \log \frac{\ph_i}{p_i}\right] \nonumber\\
    &= n\sum_{i=1}^k\left(\E\left[\ph_i \log \frac{1}{p_i}\right] - \E\left[\ph_i \log \frac{1}{\ph_i}\right]\right) \nonumber\\
    &= n\sum_{i=1}^k \left(p_i \log \frac{1}{p_i} - \E\left[\ph_i \log \frac{1}{\ph_i}\right]\right). \label{eq:ExptGSimplify}
\end{align}
Since $n\ph_i \sim \Bin(n,p_i)$, we note that 
\begin{align}\label{eq:PhiBernApprox}
    \E\left[\ph_i \log \frac{1}{\ph_i}\right] &= \E[\phi(\ph_i)] = \sum_{k=0}^n \phi(k/n)p_i^k(1-p_i)^{n-k} = B_n(\phi,p_i).
\end{align}
Substituting~\eqref{eq:PhiBernApprox} into~\eqref{eq:ExptGSimplify} we obtain the assertion.
\end{proof}
\propFProperties*
\begin{proof}
For the first assertion, we have $D(\ph_{k,p}\|p) \le \log(1+\chi^2(\ph_{k,p}\|p)) \le \chi^2(\ph_{k,p}\|p)$, where the first inequality follows from e.g.~\cite[Chapter 7]{Polyanskiy--Wu14}  and second inequality follows since $\log(1+x) \le x$ for $x \ge 0$; this was proved in~\cite{Paninski2003} (see also~\cite{Jiao--Venkat--Han--Weissman2017}).%
The second assertion follows since from \cref{prop:GFuncBernstein} $f_{k,p}(n) - f_{k,p}(n+1) = \sum_{i=1}^k \left(B_{n+1}(\phi,p_i) - B_{n}(\phi,p_i)\right)$ %
, and we have $B_{n+1}(\phi,p_i) \ge B_{n}(\phi,p_i)$ for all $0 \le p_i \le 1$ (from \cite[Lemma 4.3]{WuYang2020} and concavity of $\phi(\cdot)$).
\end{proof}
\section{Proofs Omitted From \texorpdfstring{\cref{sec:MainThmPf,sec:LemmasForProof}}{}}
\label{app:LemmasForProof}
\LemBddDiffsKL*
\begin{proof}
We can see that $Z_{n,k,p} = 2n D(\ph_{n,k}|| p)$ is a function of $Z_1,\dotsc,Z_n$, all i.i.d., where $\Prob(Z_1 = j) = p_j$. Denoting this function by $h(\cdot)$, we have 
\begin{align}
    2n D(\ph_{n,k}\| p) = h(Z_1,\dotsc,Z_n) = 2n \sum_{i=1}^k \ph_i \log \frac{\ph_i}{p_i} \nonumber .
\end{align}
Let $Z_i = j \in [k]$, and let $Z_i' = j' \in [k]$. 
We then have 
\begin{align}
    &\frac{1}{2n}\left|h(Z_1,\dotsc,Z_i,\dotsc,Z_n) - h(Z_1,\dotsc,Z_i',\dotsc,Z_n)\right| \nonumber\\
    &\quad= \left|\ph_j \log \frac{\ph_j}{p_j} + \ph_{j'}\log \frac{\ph_{j'}}{p_{j'}} - \left(\ph_j - \frac{1}{n} \right)\log \frac{\left(\ph_j-\frac{1}{n}\right)}{p_j} - \left(\ph_{j'} + \frac{1}{n} \right)\log \frac{\left(\ph_{j'}+\frac{1}{n}\right)}{p_j} \right| \nonumber\\
    &\quad= \left|\phi\left(\ph_j - \frac{1}{n}\right) - \phi(\ph_j) +  \phi(\ph_{j'}) - \phi\left(\ph_{j'} + \frac{1}{n}\right) + \frac{1}{n} \log \frac{1}{p_j} - \frac{1}{n}\log\frac{1}{p_{j'}}  \right|\label{eq:ResamplingDiff} \\
    &\quad\le \left| \phi\left(\ph_j - \frac{1}{n}\right) - \phi(\ph_j) \right| + \left| \phi(\ph_{j'}) - \phi\left(\ph_{j'} + \frac{1}{n} \right) \right| + \frac{2}{n}\log\frac{1}{\a} \label{eq:PiMin}
\end{align}
where~\eqref{eq:ResamplingDiff} follows by recalling that $\phi(x) = x \log \frac{1}{x}$ and~\eqref{eq:PiMin} follows from the triangle inequality and the fact that $p_j, p_{j'} \ge \a$.
Now, for all integers $2 \le l \le n$ we have, by concavity of $\phi(\cdot)$ that 
\begin{align*}%
   \left|\phi\left(\frac{l}{n}\right) - \phi\left(\frac{l-1}{n}\right)\right|  \le \log\frac{n}{e(l-1)}\cdot\frac{1}{n} \le \frac{\log n}{n}
\end{align*}
and also $\phi\left(\frac{1}{n}\right) - \phi(0) = \phi\left(\frac{1}{n}\right) = \frac{\log n}{n}$. Using this and the fact that $\ph_j, \ph_{j'}+\frac{1}{n} \in \left\{\frac{1}{n},\dotsc,\frac{n}{n}\right\}$ in~\eqref{eq:PiMin} we have 
\[
\frac{1}{2n}\left|h(Z_1,\dotsc,Z_i,\dotsc,Z_n) - h(Z_1,\dotsc,Z_i',\dotsc,Z_n)\right| \le \frac{2 \log(n/\a)}{n}.
\]
Now using the bounded differences inequality~\cite[Theorem 6.2]{Boucheron--Lugosi--Massart13} we have 
\[
\Prob\left(\left|Z_{n,k,p} - \E Z_{n,k,p} \right| > t\right) \le 2\exp\left(-\frac{t^2}{8n \log^2(n/\a)}\right).
\]
This implies that (see, for example,~\cite[Theorem 3.10]{duchi2016lecture} or~\cite{Boucheron--Lugosi--Massart13}) %
whenever $p \in \Delta^{k-1}_a$, for all $t$ we have

\[
   \psi_{Z_{n,k,p} - \E[Z_{n,k,p}]}(t)  \le 27 n \log^2\left(\frac{n}{\a}\right) t^2 .
\]

\end{proof}
\lemDiscGradF*
\begin{proof}
From \cref{prop:fProperties}, we have $f_{k,p}(n) \ge f_{k,p}(n+1)$ for all $n$. To prove the upper bound, we recall 
\begin{align}
f_{k,p}(n) - f_{k,p}(n+1) &=   \sum_{i=1}^k(B_{n+1}(\phi,p_i) - B_n(\phi,p_i))  \label{eq:SumOfBern}.
\end{align}
Fix an index $i \in [k]$. Let $N_i \sim \text{Bin}(n,p_i)$ and $Y_i \sim \text{Ber}(p_i)$ be two independent random variables. By definition, we have $B_n(\phi,p_i) = \Expt\left[\phi\left(\frac{N_i}{n}\right)\right]$ and $B_{n+1}(\phi,p_i) = \Expt\left[\phi\left(\frac{N_i+Y_i}{n+1}\right)\right]$. Moreover, since $\phi(x) = x\log(1/x)$ is concave on $[0,1]$, we have for $x,y \in (0,1]$ that $\phi(x) - \phi(y) \leq \phi'(y)(x- y) = \log(1/(ey)) (x-y)$. Therefore, 
\begin{align}
B_{n+1}(\phi,p_i) - B_n(\phi,p_i) &= \E\left[ \phi\left( \frac{N_i + Y_i}{n+1} \right) - \phi \left( \frac{N_i}{n} \right) \right] \nonumber\\
&= \E\left[\indic{N_i\geq 1} \left( \phi\left( \frac{N_i + Y_i}{n+1} \right) - \phi \left( \frac{N_i}{n} \right) \right) \right] + \E\left[\indic{N_i=0} \left( \phi\left( \frac{N_i + Y_i}{n+1} \right) - \phi \left( \frac{N_i}{n} \right) \right) \right] \nonumber \\
&= \E\left[\indic{N_i\geq 1} \left( \phi\left( \frac{N_i + Y_i}{n+1} \right) - \phi \left( \frac{N_i}{n} \right) \right) \right] + (1 - p_i)^n\E\left[  \left( \phi\left( \frac{Y_i}{n+1} \right) \right) \right] \label{eq:NYIndep}\\
&\leq  \E \left[\indic{N_i\geq 1} \log\left(\frac{n}{e N_i}\right) \left(\frac{N_i+Y_i}{n+1} - \frac{N_i}{n}\right)  \right] + (1 - p_i)^n  \phi\left( \frac{p_i}{n+1}  \right) \label{eq:ConcavityPhi}
\end{align}
where~\eqref{eq:NYIndep} uses the fact that $N_i \Indep Y_i$ and $\Prob(N_i = 0) = (1-p_i)^n$, and~\eqref{eq:ConcavityPhi} follows from concavity of $\phi(\cdot)$ and the Jensen inequality.  We will now focus on the first expression in~\eqref{eq:ConcavityPhi}. We have
\begin{align}
\E \left[\indic{N_i\geq 1} \log\left(\frac{n}{e N_i}\right) \left(\frac{N_i+Y_i}{n+1} - \frac{N_i}{n}\right)  \right]    			&=  \E \left[ \indic{N_i\geq 1}\log\left(\frac{n}{e N_i}\right) \left(\frac{nY_i - N_i}{n(n+1)} \right) \right] \nonumber \\
&=  \E \left[\indic{N_i\geq 1}\log\left(\frac{n}{e N_i}\right) \left(\frac{ np_i - N_i}{n(n+1)} \right) \right] \label{eq:IndepNYExp}\\
&\leq \log(n)  \E \left[ \indic{N_i\geq 1}\left|\frac{ np_i - N_i}{n(n+1)} \right| \right] \label{eq:NiGeqOne}\\
&\leq \frac{\log n}{n^2} \sqrt{\E \left[ \left| np_i - N_i \right|^2 \right]}  \label{eq:JensenNi}\\
&= \frac{\log n}{n^2} \sqrt{ np_i(1 - p_i)}  \label{eq:BinVar}
\end{align}
where~\eqref{eq:IndepNYExp} follows since $N_i \Indep Y_i$,~\eqref{eq:NiGeqOne} uses that $N_i \ge 1$,~\eqref{eq:JensenNi} uses the Jensen inequality and~\eqref{eq:BinVar} uses the fact that variance of $N_i$ is $np_i(1-p_i)$.
Summing up~\eqref{eq:BinVar} from 1 through $k$ we have 
\begin{align}
  \sum_{i=1}^k  \E \left[\indic{N_i\geq 1} \log\left(\frac{n}{e N_i}\right) \left(\frac{N_i+Y_i}{n+1} - \frac{N_i}{n}\right)  \right]    &\le \frac{\log n}{n^{3/2}} \sum_{i=1}^k \sqrt{ p_i(1 - p_i)} \nonumber \\ &\leq \frac{\log n}{n^{3/2}}\sqrt{k-1}  \label{eq:SumOfFirstTermDiscGrad}
\end{align}
where we use the Cauchy--Schwartz inequality. Focusing now on the second term in~\eqref{eq:ConcavityPhi} and defining $p^* = \frac{2 \log n}{n}$ we have
\begin{align}
    \sum_{i=1}^k (1 - p_i)^n \phi\left(\frac{p_i}{n+1}\right) &\leq \sum_{i=1}^k e^{-np_i} \phi\left(\frac{p_i}{n+1}\right) \label{eq:PowerExpIneq}\\
    &= \sum_{i: p_i \geq p^*}^k e^{-np_i} \phi\left(\frac{p_i}{n+1}\right) +\sum_{i: p_i < p^*}^k e^{-np_i} \phi\left(\frac{p_i}{n+1}\right) \nonumber \\
    &\le \sum_{i: p_i \geq p^* }^k \frac{1 }{n^{2}}  + \sum_{i: p_i < p^*}^k \phi\left(\frac{p_i}{n+1}\right) \label{eq:phiUB}\\
    &\leq \frac{k}{n^{2}} + k \phi\left(\frac{p^*}{n+1}\right) \label{eq:phiMontone}\\
    &\le \frac{k}{n^{2}} + \frac{4k \log^2 n}{n^2} \label{eq:SubstitutePStar} \\
    &\le  \frac{8k \log^2 n}{n^2}, \label{eq:secondtermsumbd}
\end{align}
where~\eqref{eq:PowerExpIneq} uses that $(1-x)^n \le e^{-nx}$;~\eqref{eq:phiUB} uses the fact that $\phi(x) \le 1$ for $x \in [0,1]$, $e^{-np_i}$ is always decreasing in $p_i$ and is $\le 1$;~\eqref{eq:phiMontone} follows since $\phi(\cdot)$ is increasing in $\left(0,\frac{p^*}{n+1}\right)$ (as $\frac{p^*}{n+1} = \frac{2 \log n}{n(n+1)} \le \frac{1}{e}$); and~\eqref{eq:SubstitutePStar} follows by substituting $p^* = 2\log n/n$ and some algebra. Finally,~\eqref{eq:SumOfFirstTermDiscGrad} and~\eqref{eq:secondtermsumbd} along with~\eqref{eq:ConcavityPhi} and~\eqref{eq:SumOfBern} yield the assertion.
\end{proof}

\resClaimAgrRawConf*

\begin{proof}
The result in \cite{Agrawal2020} implies the following: for all $t \geq 2(k-1)$
\begin{align}
\nonumber\P\{Z \geq t\} \leq \exp(-t/2) \left( \frac{e t}{2(k-1)} \right)^{k-1} &= \exp\left( - \frac{t}{2} + k \log\left( \frac{et}{2 (k-1)} \right) \right) \\
&\leq  \exp\left( - \frac{t}{2} + k + k\log\left( \frac{t}{k} \right) \right),
\label{eq:rawMomFirst}
\end{align}
where the last expression uses that $2(k-1) \geq k$.
Let $t = 2x k + 2x \log(1/\delta)  + 2 \log(1/\delta) + 2k$ for $x \geq 100$ and 
 $\gamma = \log(1/\delta)/k$. 
 We have that the expression in \eqref{eq:rawMomFirst} is bounded above by:
\begin{align*}
\exp\left( - \frac{t}{2} + k + k\log\left( \frac{t}{k} \right) \right)
 &= \delta\exp\left(-\left((x - \log 2)k + x \log(1/\delta) - k \log\left( (1 + x) + (x + 1)\gamma  \right) \right)  \right)\\ 
 &\leq \delta\exp\left(-\left( 0.5xk + x\log(1/\delta) - k \log\left( 2xk + 2x\gamma  \right) \right)  \right) \\
\qquad \qquad &= \delta\exp\left(- k \left(\underbrace{ 0.5x - \log(2x) + x \gamma - \log\left( 1 + \gamma\right)}_{A(\gamma ,x)} \right)  \right).
\end{align*}
We will now show that $A(\gamma,x) \geq 0$ for $ \gamma \geq 0$ whenever $x \geq 100$.
Using that $0.25x \geq \log(2x)$ and $x \geq 1$ and $\log(1 + \gamma) \leq \gamma$, we have that 
\begin{align*}
A(\gamma,x) &=  0.5x - \log(2x) + x \gamma - \log\left( 1 + \gamma\right) \geq 0.25x +  \gamma - \log(1 + \gamma) \geq 0.
\geq 0.25x \geq 0.\end{align*}
Taking $x = 100$,  we thus have that with probability $1 - \delta$,
$Z \leq t = 2(x+1)k + 2(x + 1)\log(1/\delta) \leq 4x k + 4x \log(1/\delta)  $, completing the proof with $C = 400$.
\end{proof}

 \bibliographystyle{IEEEtran}
 \bibliography{KLConc.bib}

\begin{thebibliography}{10}
\providecommand{\url}[1]{#1}
\csname url@samestyle\endcsname
\providecommand{\newblock}{\relax}
\providecommand{\bibinfo}[2]{#2}
\providecommand{\BIBentrySTDinterwordspacing}{\spaceskip=0pt\relax}
\providecommand{\BIBentryALTinterwordstretchfactor}{4}
\providecommand{\BIBentryALTinterwordspacing}{\spaceskip=\fontdimen2\font plus
\BIBentryALTinterwordstretchfactor\fontdimen3\font minus
  \fontdimen4\font\relax}
\providecommand{\BIBforeignlanguage}[2]{{%
\expandafter\ifx\csname l@#1\endcsname\relax
\typeout{** WARNING: IEEEtran.bst: No hyphenation pattern has been}%
\typeout{** loaded for the language `#1'. Using the pattern for}%
\typeout{** the default language instead.}%
\else
\language=\csname l@#1\endcsname
\fi
#2}}
\providecommand{\BIBdecl}{\relax}
\BIBdecl

\bibitem{Lehmann--Romano2006}
E.~L. Lehmann and J.~P. Romano, \emph{Testing statistical hypotheses}.\hskip
  1em plus 0.5em minus 0.4em\relax Springer Science \& Business Media, 2006.

\bibitem{Nowak--Tanczos2019}
R.~Nowak and E.~T{\'a}nczos, ``Tighter confidence intervals for rating
  systems,'' \emph{arXiv preprint arXiv:1912.03528}, 2019.

\bibitem{Malloy--Tripathy--Nowak2020}
M.~L. Malloy, A.~Tripathy, and R.~D. Nowak, ``Optimal confidence regions for
  the multinomial parameter,'' \emph{arXiv preprint arXiv:2002.01044}, 2020.

\bibitem{VanDerVaart2000}
A.~W. Van~der Vaart, \emph{Asymptotic statistics}.\hskip 1em plus 0.5em minus
  0.4em\relax Cambridge university press, 2000, vol.~3.

\bibitem{Polyanskiy--Wu14}
Y.~Polyanskiy and Y.~Wu, \emph{Lecture notes on information theory}.\hskip 1em
  plus 0.5em minus 0.4em\relax Available Online, 2014.

\bibitem{Csiszar--Shields2004}
I.~Csisz{\'a}r and P.~C. Shields, ``Information theory and statistics: A
  tutorial,'' \emph{Foundations and Trends{\textregistered} in Communications
  and Information Theory}, vol.~1, no.~4, pp. 417--528, 2004.

\bibitem{Mardia--Jiao--Tanczos--Nowak--Weissman2020}
J.~Mardia, J.~Jiao, E.~T{\'a}nczos, R.~D. Nowak, and T.~Weissman,
  ``Concentration inequalities for the empirical distribution of discrete
  distributions: beyond the method of types,'' \emph{Information and Inference:
  A Journal of the IMA}, vol.~9, no.~4, pp. 813--850, 2020.

\bibitem{Agrawal2020}
R.~Agrawal, ``Finite-sample concentration of the multinomial in relative
  entropy,'' \emph{IEEE Transactions on Information Theory}, vol.~66, no.~10,
  pp. 6297--6302, 2020.

\bibitem{Boucheron--Lugosi--Massart13}
S.~Boucheron, G.~Lugosi, and P.~Massart, \emph{Concentration inequalities: A
  nonasymptotic theory of independence}.\hskip 1em plus 0.5em minus 0.4em\relax
  Oxford university press, 2013.

\bibitem{Agrawal2020Thesis}
R.~Agrawal, ``Deriving indistinguishability from unpredictability: Tools and
  applications in pseudorandomness,'' Ph.D. dissertation, Harvard University,
  2020.

\bibitem{10.5555/1146355}
T.~M. Cover and J.~A. Thomas, \emph{Elements of Information Theory (Wiley
  Series in Telecommunications and Signal Processing)}.\hskip 1em plus 0.5em
  minus 0.4em\relax USA: Wiley-Interscience, 2006.

\bibitem{Csiszar1998}
I.~Csisz{\'a}r, ``The method of types,'' \emph{IEEE Transactions on Information
  Theory}, vol.~44, no.~6, pp. 2505--2523, 1998.

\bibitem{Guo--Richardson2020}
F.~R. Guo and T.~S. Richardson, ``Chernoff-type concentration of empirical
  probabilities in relative entropy,'' \emph{IEEE Transactions on Information
  Theory}, vol.~67, no.~1, pp. 549--558, 2020.

\bibitem{AntosKontoyiannis2001}
A.~Antos and I.~Kontoyiannis, ``\BIBforeignlanguage{en}{Convergence properties
  of functional estimates for discrete distributions},''
  \emph{\BIBforeignlanguage{en}{Random Structures \& Algorithms}}, vol.~19, no.
  3-4, pp. 163--193, 2001.

\bibitem{vershynin2018high}
R.~Vershynin, \emph{High-dimensional probability: An introduction with
  applications in data science}.\hskip 1em plus 0.5em minus 0.4em\relax
  Cambridge university press, 2018.

\bibitem{Castellan2003}
G.~Castellan, ``Density estimation via exponential model selection,''
  \emph{IEEE transactions on information theory}, vol.~49, no.~8, pp.
  2052--2060, 2003.

\bibitem{Massart2006}
P.~Massart, \emph{Concentration inequalities and model selection}.\hskip 1em
  plus 0.5em minus 0.4em\relax Springer, 2007.

\bibitem{Kamath--Orlitsky--Pichapati--Suresh2015}
S.~Kamath, A.~Orlitsky, D.~Pichapati, and A.~T. Suresh, ``On learning
  distributions from their samples,'' in \emph{Conference on Learning
  Theory}.\hskip 1em plus 0.5em minus 0.4em\relax PMLR, 2015, pp. 1066--1100.

\bibitem{Simon2006}
M.~K. Simon, \emph{\BIBforeignlanguage{en}{Probability {{Distributions
  Involving Gaussian Random Variables}}}}.\hskip 1em plus 0.5em minus
  0.4em\relax {Boston, MA}: {Springer US}, 2006.

\bibitem{BousquetEtAl2019}
O.~Bousquet, Y.~Klochkov, and N.~Zhivotovskiy, ``Sharper bounds for uniformly
  stable algorithms,'' in \emph{Proceedings of Thirty Third Conference on
  Learning Theory}, vol. 125.\hskip 1em plus 0.5em minus 0.4em\relax PMLR,
  09--12 Jul 2020.

\bibitem{duchi2016lecture}
J.~Duchi, ``Lecture notes for statistics 311/electrical engineering 377,''
  \emph{URL: https://stanford. edu/class/stats311/Lectures/full\_notes.pdf.},
  vol.~2, p.~23, 2016.

\bibitem{braess2004bernstein}
D.~Braess and T.~Sauer, ``Bernstein polynomials and learning theory,''
  \emph{Journal of Approximation Theory}, vol. 128, no.~2, pp. 187--206, 2004.

\bibitem{WuYang2020}
Y.~Wu and P.~Yang, ``\BIBforeignlanguage{English}{Polynomial {{Methods}} in
  {{Statistical Inference}}: {{Theory}} and {{Practice}}},''
  \emph{\BIBforeignlanguage{English}{Foundations and Trends\textregistered{} in
  Communications and Information Theory}}, vol.~17, no.~4, pp. 402--586, Oct.
  2020.

\bibitem{Paninski2003}
L.~Paninski, ``\BIBforeignlanguage{en}{Estimation of {{Entropy}} and {{Mutual
  Information}}},'' \emph{\BIBforeignlanguage{en}{Neural Computation}},
  vol.~15, no.~6, pp. 1191--1253, Jun. 2003.

\bibitem{Mitzenmacher--Upfal2017}
M.~Mitzenmacher and E.~Upfal, \emph{Probability and computing: Randomization
  and probabilistic techniques in algorithms and data analysis}.\hskip 1em plus
  0.5em minus 0.4em\relax Cambridge university press, 2017.

\bibitem{Knoblauch2008}
A.~Knoblauch, ``Closed-form expressions for the moments of the binomial
  probability distribution,'' \emph{SIAM Journal on Applied Mathematics},
  vol.~69, no.~1, pp. 197--204, 2008.

\bibitem{Jiao--Venkat--Han--Weissman2017}
J.~Jiao, K.~Venkat, Y.~Han, and T.~Weissman, ``Maximum likelihood estimation of
  functionals of discrete distributions,'' \emph{IEEE Transactions on
  Information Theory}, vol.~63, no.~10, pp. 6774--6798, 2017.

\end{thebibliography}

\end{document}